\documentclass{article}

\usepackage[centertags]{amsmath}
\usepackage{latexsym}
\usepackage{amsfonts, amssymb, amsthm}
\usepackage{mathrsfs}
\usepackage{hyperref}
\usepackage[utf8]{inputenc}
\usepackage{url}
\usepackage{enumerate}
\usepackage[blocks]{authblk}
\usepackage{tikz}
\usetikzlibrary{shapes}
\usepackage{tabto}

\usepackage[a4paper]{geometry}

\theoremstyle{plain}
\newtheorem{thm}{Theorem}[section]
\newtheorem{lemma}[thm]{Lemma}
\newtheorem{prop}[thm]{Proposition}

\theoremstyle{definition}
\newtheorem{defn}[thm]{Definition}

\theoremstyle{remark}
\newtheorem{rema}[thm]{Remark}
\newtheorem{conjecture}[thm]{Conjecture}

\def\Q{{\mathbf Q}}
\def\R{{\mathbf R}}
\def\Z{{\mathbf Z}}

\def\C{{\mathbf C}}
\def\OO{\mathcal{O}}

\def\m{\mathfrak m}
\newcommand{\F}{\mathbf{F}}

\newcommand{\map}[4]{
  \begin{array}{ccc}
    #1 & \longrightarrow & #2 \\
    #3 & \longmapsto      & #4
  \end{array}
}
\renewcommand{\geq}{\geqslant}
\renewcommand{\leq}{\leqslant}

\title{Towards a function field version of Freiman's Theorem}

\newcommand*\samethanks[1][\value{footnote}]{\footnotemark[#1]}
\title{Towards a function field version of Freiman's Theorem}
\author{Christine Bachoc\thanks{
Institut de Mathématiques de Bordeaux, CNRS UMR
    5251, Université de Bordeaux, 351 cours de la Libération, 33400 Talence.} and
  Alain Couvreur\thanks{INRIA \& Laboratoire LIX, CNRS UMR 7161,
    \'Ecole Polytechnique, Université Paris Saclay, 
    91128 Palaiseau Cedex} and
  Gilles Z\'emor\samethanks[1]}

\begin{document}

\maketitle

\begin{abstract}
  We discuss a multiplicative counterpart of Freiman's $3k-4$ theorem
  in the context of a function field $F$ over an algebraically closed
  field $K$. Such a theorem would
  give a precise description of subspaces $S$, such that the space
  $S^2$ spanned by products of elements of $S$ satisfies
  $\dim S^2 \leq 3 \dim S-4$. We make a step in this direction by
  giving a complete characterisation of spaces $S$ such that
  $\dim S^2 = 2 \dim S$.  We show that, up to multiplication by a
  constant field element, such a space $S$ is included in a function
  field of genus $0$ or $1$.  In particular if the genus is $1$ then
  this space is a Riemann-Roch space.
\end{abstract}


\section{Introduction}

We are interested in linear analogues of addition theorems that occur
in field extensions $F/K$ of a base field $K$. If $S$ and $T$ are
finite-dimensional $K$-vector subspaces of $F$, we denote by $ST$
the $K$-linear span of the set of all products $st$, $s\in S$, $t\in T$.
The general purpose of this area of research is to characterise
subspaces $S$ and $T$ whose product $ST$ has unusually small
dimension: it is naturally inspired by one of the goals of additive combinatorics 
which is to characterise subsets $A,B$ of elements of a group that have
sumsets $A+B$ of small cardinality, where $A+B$ denotes the set of
elements $a+b$, $a\in A$, $b\in B$.

The first significant result in this direction is arguably due to
Hou, Leung and Xiang \cite{hlx02} and generalises the classical
addition theorem of Kneser \cite{kneser}. It essentially states that if
$\dim ST < \dim S + \dim T -1$, then the space $ST$ must be stabilised
by a non-trivial subfield of $F$.
A welcome feature of Hou et
al.'s theorem is that Kneser's original theorem can be recovered from
it, so that it is not only a transposition to the linear setting of
its additive counterpart, but it can also be seen as a generalisation.
Hou's Theorem was finally proved for all field extensions in
\cite{bszKneser}, and also studied in other algebras than field
extensions \cite{bl15,mz15}.
Linear versions of addition theorems were also studied in the somewhat
broader context of skew field extensions in
\cite{el09}. {Many applications of the theory of
  products of spaces in the algebra $\F_q^n$ with componentwise
  multiplication are discussed in \cite{hugues}.} 

A common feature of {many of} the above works is that they tend to focus on
highlighting the existence of finite dimensional subfields or
subalgebras, whenever $\dim ST < \dim S + \dim T -1$.  In contrast, in
\cite{bszVosper} field extensions $F/K$ are studied where there are no
subextensions of $K$ in $F$ ($K$ is algebraically closed in $F$) in
which case one always has $\dim ST \geq \dim S +\dim T -1$ whenever
$ST\neq F$ \cite{el09}. The goal of \cite{bszVosper} was to prove that
the equality $\dim ST = \dim S +\dim T -1$ essentially implies that
$S$ and $T$ have bases in geometric progression: this is a linear
equivalent of Vosper's Theorem \cite{Vosper} which states that in a
group of prime order, or more generally an abelian group with no
finite subgroups, $|A+B|=|A|+|B|-1$ implies that $A$ and $B$ are
arithmetic progressions (with some degenerate cases ruled out).  It is
proved in \cite{bszVosper} that a linear version of Vosper's theorem
holds when the base field $K$ is finite, and for a number of other
base fields, but not for every field
$K$, even if it is assumed to be algebraically closed in $F$.
{This theory of
  subspaces with products of small dimension in field extensions has
  also recently found applications to coding theory in \cite{RRT18}.}

A particularly simple case for which a linear version of Vosper's
Theorem can be derived from an addition theorem, is when the base
field $K$ is itself assumed to be algebraically closed. The linear
theorem then follows almost directly from considering sets $A$ and $B$ of 
valuations of the field elements in $S$ and $T$ and arguing that $A$ 
and $B$ must satisfy an addition theorem. From this perspective it
becomes very natural to ask what can be said of the structure of
spaces $S$ such that
\begin{equation}
  \label{eq:cgenus}
  \dim S^2 =2\dim S -1 +\gamma
\end{equation}
for increasing values of $\gamma$. We have switched to the symmetric
situation $S=T$ for the sake of simplicity.

In the additive case, recall Freiman's
``$3k-4$'' Theorem \cite{Freiman}, \cite[Th. 5.11]{TaoVu}, which says
that in a torsion-free abelian group,
  $$|A+A|=2|A|-1+\gamma$$
  implies, when $\gamma\leq |A|-3$, that $A$ is included in an
  arithmetic progression of length $|A|+\gamma$ (i.e. $A$ is a
  progression with at most $\gamma$ missing elements).
  {The full Freiman Theorem, which extends the above
    $3k-4$ version, is arguably a cornerstone of additive
    combinatorics and has inspired a lot of subsequent work (see
    e.g. \cite{TaoVu}).  In this light, tackling the characterisation
    of spaces satisfying \eqref{eq:cgenus} would be a welcome addition
    to the burgeoning theory of space products in extension fields.}

Candidates for spaces $S$ satisfying \eqref{eq:cgenus} are of course
subspaces (of codimension at most $\gamma$) inside a space that has a
basis in geometric progression.  However, some thought yields
alternative spaces that do not have an additive analogue when
$\gamma\geq 1$: namely Riemann-Roch spaces $L(D)$ of an algebraic
curve of genus $\gamma$, which {can be} seen to
satisfy \eqref{eq:cgenus}.  It is tempting to conjecture that, in the
case when the base field $K$ is algebraically closed, any space
satisfying \eqref{eq:cgenus} {with
  $\gamma\leq\dim S-3$} is, up to multiplication by a constant, a
subspace of codimension $t$ inside a Riemann-Roch space of an
algebraic curve of genus $g$, with $t+g\leq \gamma$. With this in
mind, let us call the quantity $\gamma$ in \eqref{eq:cgenus} the {\em
  combinatorial genus} of $S$.  In the present paper we make a modest
contribution towards this hypothesis by proving it in the case when
$\gamma=1$.

We will use a blend of combinatorial and algebraic methods. The paper
is organised as follows: {Section~\ref{sec:examples}
  starts with a discussion of concrete examples of spaces with small
  products.}  Section~\ref{sec:valuations} recalls basic properties of
valuations that will in particular associate sets of integers with
small sumsets to subspaces with products of small dimension.
Section~\ref{sec:freiman} proves
{Theorem~\ref{thm:freiman_field}}, an extension field
version of ``Freiman's Lemma'' where the transcendence degree plays a
role analogous to the rank of a set of elements of a torsion-free
abelian group.  Section~\ref{sec:products} introduces a lattice of
subspaces that we shall rely on heavily, and illustrates its
usefulness by characterising spaces with combinatorial genus equal to
zero.  Section~\ref{sec:divisors} recalls basic properties of
Riemann-Roch spaces and states the paper's main result,
Theorem~\ref{thm:main}.

Section~\ref{sec:proof_main} proves Theorem~\ref{thm:main}.
Section~\ref{sec:genus0} complements Theorem~\ref{thm:main}
by giving a precise characterisation of those subspaces of
Riemann-Roch spaces that have combinatorial genus equal to $1$.
{
Finally, Section~\ref{sec:perfect} extends Theorems~\ref{thm:freiman_field}
and~\ref{thm:main} to
the case when the base field is perfect rather than algebraically
closed.
}

\section{{Motivating Examples}}\label{sec:examples}
{
Let $K$ be a field and consider the field $F=K(x)$ of rational
functions over $K$. Suppose we want a $K$-vector subspace $S$ of
dimension $k$ such that $S^2$ has the smallest possible dimension.
A natural candidate is the space $S$ generated by the geometric
progression $1,x,x^2,\ldots ,x^{k-1}$ for which we have $\dim
S^2=2k-1$. 
We notice that the set $A$ of degrees of the rational functions 
(in this example polynomials in $x$) of $S$ is an arithmetic
progression. More generally, the set of degrees of the functions in
$S^2$ must contain $A+A$, so that $\dim S^2\geq |A+A|$. This remark
may be used to claim that if $\dim S^2$ is the smallest possible,
namely $2k-1$, then
$|A+A|$ must be as small as possible, which implies that $A$ must be
an arithmetic progression of integers, from which it is fairly straightforward to
deduce that $S$ must have a basis of elements in geometric
progression.
We will make the point below that this line of reasoning extends to
other extension fields $F$ of $K$, provided we have {\em valuations}
at our disposal to generalise degrees of rational functions.}

{
Next relax slightly the condition on $\dim S^2$ to $\dim S^2 \leq 2k$.
To construct examples of such spaces we may consider $S$ in the rational
function field $K(x)$
generated by $1,x^2,x^3,\ldots ,x^k$. These spaces are directly inspired from
the sets of integers $A=\{0,2,3,\ldots ,k\}$ such that $|A+A|=2|A|$.
However, we have additional examples of such spaces that have no
direct additive counterpart: take $K$ to be the field of complex
numbers (say) and take $F$ to be the algebraic extension $K(x,y)$ of
the rational function field $K(x)$ where $y$ satisfies the equation
$y^2-x^3+x=0$. Now consider the space $S$ of dimension $5$ generated by
$1,x,y,x^2,xy$. It is readily checked that we have $\dim S^2=2\dim S$.
The space $S$ is an example of a Riemann-Roch space of the algebraic curve
of equation $y^2-x^3+x=0$ which is elliptic or of genus $1$. Our main
result, namely Theorem~\ref{thm:main}, will tell us that these two
examples are in some sense generic. This has motivated the following
definition, and also the conjecture below:
}
\begin{defn}\label{def:gamma}
  {
  Let $K$ be a field and $F$ be a $K$--algebra. Let $S \subset F$ be a 
  finite dimensional $K$--subspace of $F$. The {\em combinatorial
  genus of $S$} is defined as the integer $\gamma$ such that
  $$
  \dim S^2 = 2\dim S - 1 + \gamma.
  $$
  }
\end{defn}


\begin{conjecture}\label{conjecture}
{
  Let $K$ be an algebraically closed field and let $F$ be an extension
  field of $K$. Let $S$ be a
  $K$-subspace of finite dimension in $F$ such that $K\subset S$. Let the combinatorial genus
  $\gamma$ of $S$ satisfy $\gamma\leq \dim S -3$. Then 
  the genus $g$
  of the field $K(S)$ satisfies $g\leq\gamma$ and there exists a Riemann-Roch space
  $L(D)$ that contains $S$ and such that
  $\dim L(D)\leq\dim S +\gamma -g$.
}
\end{conjecture}

{The next section recalls some background on valuations
 with which we will derive our first results on spaces with small
 combinatorial genus.
}
\section{{Function fields, valuations}}\label{sec:valuations}

We start by recalling some basic facts about 
valuations in function fields that will be crucial to transferring
additive statements to the extension field setting.
We refer the reader to
\cite[Ch. 6]{Bourbaki_AlgComm5-6} for further details.

Let $K$ be a field, a {\em function field in $m$ variables over $K$}
is a field $F$ which is a finitely generated algebra of transcendence
degree $m$. Equivalently it is a finite extension of the field
$K(X_1, \ldots, X_m)$ of rational functions in $m$ variables.

We recall that such fields have {\em valuations} that map $F^{\times}$
to the elements of some ordered group. Valuations are multiplicative,
i.e. $v(xy) = v(x) + v(y)$ and satisfy the ultrametric inequality,
$v(x+y) \geq \min \{v(x), v(y)\}$ with equality when $v(x) \neq v(y)$.
The map $v$ is extended to $F$ with the convention $v(0)=\infty$.



A valuation comes with {\em
  a valuation ring} $\OO \subset F$ {which is defined
  as the set of functions of non-negative valuation, together with
  addition and multiplication inherited from $F$.
  A valuation ring has a unique maximal ideal $\m$ equal to the set of
  elements of positive valuation. The quotient $\OO/\m$ is called the
  {\em residue field} of the valuation ring.}

Let $S \subseteq F$ be a finite dimensional $K$--vector space of
dimension $n>0$.  Given a valuation $v$ on $F$ with residue field $K$,
we denote by $v(S)$ the set of valuations of the non-zero elements of $S$. We
recall the following classical result, and give a proof for the sake
of self--containedness.

\begin{prop}\label{prop:filtered_basis}
  The set $v(S)$ is finite and its cardinality equals $\dim S$.
  Moreover, there exists a basis $(e_1, \ldots, e_n)$ of $S$ such that
  $$v(e_1) > v(e_2) > \ldots > v(e_n) \quad {\rm and} \quad
  \{v(e_1), \ldots, v(e_n)\} = v(S).$$ Such a basis is referred to as
  a {\em filtered basis}. {
    In addition, $S$ has a natural filtration
  $$
  \{0\} \subset S_1 \subset \cdots \subset S_{n-1} \subset S_n
  $$
  such that
  \begin{equation}\label{eq:ineq_valuations}
  \min v(S_1) > \cdots > \min v(S_{n-1}) > \min v (S).
  \end{equation}
  For every $i=1\ldots n$, the space $S_i$ is generated by $e_1,\ldots
  ,e_i$, but the filtration is uniquely defined and does not depend on the choice of a filtered basis.
  }
\end{prop}

\begin{proof}
  First notice that elements of $F$ with distinct finite valuations are
  linearly independent.
  Indeed, if $x_1, \ldots, x_k$ have distinct valuations, then so do
  $a_1x_1,\ldots ,a_kx_k$ for non-zero $a_i\in K$, 
since non-zero elements of $K$ have valuation $0$,
  so that $v(a_1 x_1 + \cdots + a_k x_k)=\min(v(x_1),\ldots v(x_k))$
  must be finite, meaning 
  $a_1 x_1 + \cdots + a_k x_k$ must be non-zero.
   This shows that $|v(S)| \leq \dim S$.
   
  Now let $E$ be a subspace of $S$ such that $v(E)=v(S)$ and suppose
  $E\subsetneqq S$. Let $x$ be an element of $S\setminus E$ with
  maximal valuation in $\{v(s), s\in S\setminus E\}$. Let $e\in E$ be
  such that $v(e)=v(x)$.  Then, $xe^{-1} \in \OO^\times$
  and since the residue field is $K$, there exists $\lambda \in K$
  such that $xe^{-1} \equiv \lambda \mod \m$. Therefore, 
  $x  -\lambda e$ has a valuation larger than $v(x)$, a
  contradiction. Therefore $E=S$, meaning that we have $\dim
  S=|v(S)|$. Choosing any $n$ elements of $S$ with distinct finite
  valuations yields a filtered basis.

  {
  Finally, the filtration $S_1 \subset \cdots \subset S_n = S$ is iteratively constructed as follows,
  $$
  S_{i-1} = \{x \in S_i ~|~ v(x) > \min v(S_i)\}.
  $$
note that this definition is independent of the choice of a filtered
basis, however one checks easily that $S_{i-1}$ is spanned by $e_1, \ldots, e_{i-1}$.
  This shows that the space $S_{i-1}$ has codimension $1$ in $S_i$ and the sequence 
  of inequalities~(\ref{eq:ineq_valuations}) follows immediately from the definition of the $S_i$'s.
  }
\end{proof}

The following lemma is elementary but fundamental to the study of
the structure of products $S^2$ of small dimension.
It enables us to involve theorems from additive combinatorics.

\begin{lemma}\label{lem:v(H^2)}
For any valuation $v$ on $F$, and any $K$--subspace $S$,
\begin{equation}\label{eq:valuation_product}
v(S)+v(S) \subseteq v(S^2).
\end{equation}  
\end{lemma}

\begin{rema}
Note that this inclusion is not necessarily an equality. For instance,
consider the subspace $H$ of $K(x)$ of basis
$1, x, x^2, x^3 + \frac 1 x$ and the valuation $v$ at infinity. Then
$v(H) = \{0,-1,-2,-3\}$ while
$H^2 = \langle \frac 1 x, 1, x, x^2, \ldots, x^5, x^6+\frac 1 {x^2}
\rangle$
whose valuation set contains $v(1/x) = 1$ which is not in $v(H)+v(H)$.  
\end{rema}

{From now on and
  until the end of Section~\ref{sec:genus0} (with a temporary exception in
  \S~\ref{subsec:divisors}), we suppose that the base field $K$
  is algebraically closed. Note that this assumption entails that any valuation on $F$
  has residue field $K$, which will therefore enable us to apply
  Proposition~\ref{prop:filtered_basis}. 
Only in Section~\ref{sec:perfect} will we consider what becomes of our
results when the base field $K$ is not algebraically closed.
}

\section{Transposing Freiman's Lemma in field extensions}
\label{sec:freiman}
Recall the following result of Freiman \cite{Freiman}, named
``Freiman's Lemma'' by Tao and Vu \cite[Lemma 5.13]{TaoVu}.

\begin{thm}\label{thm:freiman}
  Let $A$ be a finite subset of $\R^d$ such that no hyperplane of
  $\R^d$ contains a translate of $A$. Then 
  $$|A+A| \geq (d+1)|A| -d(d+1)/2.$$
\end{thm}

Let $S$ be a $K$-vector space inside a field extension $L$ of a field $K$.
We remark that for a non-zero element $s$ of $S$, the field
subextension $K(Ss^{-1})$ of $L$ is independent of the choice of the
element $s$. Let us call the {\em transcendence degree of $S$} the
transcendence degree of $K(Ss^{-1})$ over $K$. Similarly, by the {\em
  genus of $S$} we will mean the genus of the field extension
$K(Ss^{-1})/K$
(see Theorem~\ref{thm:main} in Section~\ref{sec:Vosper+1}).

We have the extension field analogue of Theorem~\ref{thm:freiman}.

\begin{thm}\label{thm:freiman_field}
  Let $K$ be {an algebraically closed}
  field and let $F \supseteq K$ be
  an extension field of $K$. Let $S$ be a $K$-vector
  subspace of $F$ of finite dimension and of transcendence degree $d$.
  Then
  $$\dim S^2\geq (d+1)\dim S -d(d+1)/2.$$
\end{thm}


The proof of Theorem~\ref{thm:freiman_field} rests upon the following lemma.

\begin{lemma}\label{lem:val}
  If $F/K$ is a field extension over an algebraically closed field
  $K$, and if $x_1,\ldots ,x_d$ are $K$-algebraically independent
  elements of $F$ such that $F$ is an algebraic extension of
  $K(x_1,\ldots,x_d)$, then there exists a valuation $v$ of~$F$, such
  that the associated valuation ring has residue field isomorphic to
  $K$ and such that the valuation values $v(x_1),v(x_2),\ldots,v(x_d)$
  generate a group isomorphic to~$\Z^d$.
\end{lemma}

\begin{proof}
  It is standard to construct a valuation $v$ from $K(x_1,x_2,\ldots
  ,x_d)$ to $\Z^d$ such that $v(x_1)=(1,0,\ldots,0),
  v(x_2)=(0,1,\ldots,0),\ldots, v(x_d)= (0,\ldots,0,1)$ and with
  associated residue field isomorphic to $K$ (see
  e.g. \cite[Ch. 6, \S 3.4, Example~6]{Bourbaki_AlgComm5-6}).
  This valuation can then be extended to the whole of $F$
  \cite[Ch. 6, \S 3,3, Proposition~5]{Bourbaki_AlgComm5-6} with its
  residue field necessarily becoming an algebraic extension of the original
  residue field associated to $v$
  \cite[Ch. 6, \S 8.1, Proposition 1]{Bourbaki_AlgComm5-6}. Since~$K$
  is algebraically closed, the
  residue field associated to the extended valuation must 
  therefore also be isomorphic to $K$.
\end{proof}

\begin{proof}[Proof of Theorem~\ref{thm:freiman_field}]
Without loss of
generality we may suppose $K\subset S$ and $F=K(S)$. 

Let $x_1,\ldots ,x_d$
be $d$ algebraically independent elements of $F$.
Choose for $v$ a valuation given by
Lemma~\ref{lem:val}.  Then, since the residue field associated to $v$
is $K$, by Proposition~\ref{prop:filtered_basis} we know that
$\dim S=|v(S)|$ and $\dim S^2=|v(S^2)|$.
From~(\ref{eq:valuation_product}) we have $\dim S^2 \geq |v(S)+v(S)|$
and Theorem~\ref{thm:freiman} now gives us
$$|v(S)+v(S)|\geq (d+1)|v(S)|-d(d+1)/2 = (d+1)\dim S-d(d+1)/2$$
which proves the theorem.
\end{proof}


\noindent {\bf Consequence.} When one considers a space $S$, $K\subset
S \subset F$,
with $\dim S^2 \leq 3\dim S - 4$, and $F = K(S)$, then $F$ is a
function field in one variable. In particular,
from  \cite[Theorem 1.1.16]{Stichtenoth}, every valuation on $F$
is discrete and its set of  values is $\Z$.

\medskip

For the rest of this article we will assume this setting, namely a
sufficiently small combinatorial genus $\gamma$, so that the
transcendence degree of $S$ can only be equal to $1$. 
The term {\em function field} will consequently always 
mean from now on {\em function field in one variable}.
Since the multiplicative properties of $S$ that we wish to study are
invariant by multiplication by a constant non-zero element, it will
also be convenient to systematically assume $1\in S$, so that $K(S)$
is a function field (in one variable). 

\section{Products of spaces, the lattice of subspaces and
  characterising spaces with combinatorial genus $\gamma=0$}
\label{sec:products}

In order to study the structure of a
product set $S^2$, the lattice of subspaces that we introduce below
will be particularly useful.  Its structure will enable us to almost
immediately characterise spaces with the smallest possible
combinatorial genus.

\subsection{The lattice of subspaces}\label{sec:lattice}

Let $(e_1,\ldots ,e_n)$ be a filtered basis of the space $S$
relative to a valuation $v$.
Consider the sequence of
subspaces {introduced in
Proposition~\ref{prop:filtered_basis}}
$$S_1\subset S_2\subset\cdots \subset S_n=S$$
where $S_i$ denotes the subspace of $S$ generated by $e_1,\ldots
,e_i$. We will refer to this sequence of spaces as the {\em
  filtration of $S$ relative to $v$}. 

Since dimensions of spaces are unchanged by multiplication by a
constant element, we may assume $e_1=1$ and $S_1=K$.
It will be useful to consider the lattice of subspaces of $S^2$
consisting of the products of subspaces $S_iS_j$ and ordered by
inclusion, as represented on Figure~\ref{fig:subspaces}. 
We will consider directed edges between $S_iS_j$ and $S_iS_{j+1}$ and
between $S_iS_j$ and $S_{i+1}S_{j}$, and label both edges by a weight
defined as the
codimension of $S_iS_j$ inside $S_iS_{j+1}$ and $S_{i+1}S_{j}$
respectively.
We will make several times the argument that the sum of weights on two
directed paths that lead from the same initial vertex to the same
terminal vertex must be the same because they both equal the
codimension of the initial subspace inside the terminal subspace.
Notice also that all weights must be positive because the valuation
set of an initial subspace must be strictly smaller than the valuation
set of the corresponding terminal subspace: indeed, $e_ie_{j+1}$ is an
element of minimal valuation of $S_iS_{j+1}$ that cannot belong to
$S_iS_j$ because the minimum of $v(S_iS_j)$ is attained by $e_ie_j$,
and similarly for $S_{i+1}S_j$.

\begin{figure}[!h]
  \centering
\begin{tikzpicture}[scale=0.8]
\node[name=S11] at (4,4) {$K=S_1^2$};
\node[name=S12] at (6,4) {$S_{1}S_{2}$};
\node[name=S13] at (8,4) {$S_1S_3$};
\node[name=S14] at (10,4) {$S_1S_4$};
\node[name=S15] at (12,4) {$S_1S_5$};
\node[name=S22] at (6,6) {$S_2^2$};
\node[name=S23] at (8,6) {$S_2S_3$};
\node[name=S24] at (10,6) {$S_2S_4$};
\node[name=S25] at (12,6) {$S_2S_5$};
\node[name=S33] at (8,8) {$S_3^2$};
\node[name=S34] at (10,8) {$S_3S_4$};
\node[name=S35] at (12,8) {$S_3S_5$};
\node[name=S44] at (10,10) {$S_4^2$};
\node[name=S45] at (12,10) {$S_4S_5$};
\node[name=S55] at (12,12) {$S_5^2$};
\draw[->] (S11) -- (S12) edge (S13) (S13) edge (S14) (S14) edge (S15) (S15)
edge (S25) (S25) edge (S35) (S35) edge (S45) (S45) edge (S55);
\draw[->] (S22) -- (S23) edge (S24) (S24) edge (S25);
\draw[->] (S33)--(S34) edge (S35);
\draw[->] (S14) to (S24) edge (S34) (S34) edge (S44) (S44) edge (S45);
\draw[->] (S12) edge (S22);
\draw[->] (S13) -- (S23) edge (S33);
\end{tikzpicture}
  \caption{the lattice of subspaces}
  \label{fig:subspaces}
\end{figure}
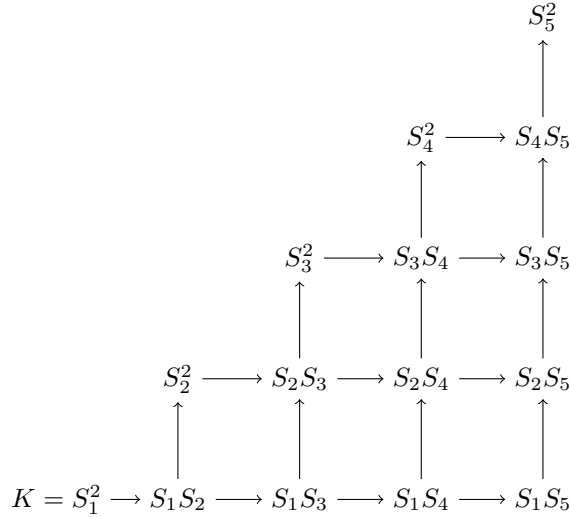

The following lemma states that when two edges that fall into the same
terminal vertex both have weight $1$, then the initial vertices
correspond to the same subspace.

\begin{lemma}\label{lem:codim1}
  Suppose the spaces $S_iS_{j+1}$ and $S_{i+1}S_j$ both have codimension $1$
  inside $S_{i+1}S_{j+1}$, then $S_iS_{j+1} = S_{i+1}S_j$.
\end{lemma}

\begin{proof}
  Let $U=\{s\in S_{i+1}S_{j+1}, v(s)>\min v(S_{i+1}S_{j+1})\}$. We
  have that $U\subsetneqq S_{i+1}S_{j+1}$ and $U$ is a subspace
  containing both $S_iS_{j+1}$ and $S_{i+1}S_j$ which must therefore
  all have the same dimension and be equal.
\end{proof}

\subsection{The structure of $S$ when $\gamma=0$}
\label{sec:Vosper}

Like in the previous subsection we assume that $e_1=1$, and we moreover set $x=e_2$.

\begin{lemma}\label{lem:vosper}
  Suppose all directed edges lying on any path from $S_1$ to $S_iS_j$,
  $2\leq i$, have weight~$1$. Then for every $k$, $2\leq k\leq j$, the space $S_k$ is generated
  by $1,x,x^2,\ldots ,x^{k-1}$.
\end{lemma}

\begin{proof}
Proceed by induction on $k$. Suppose the result is
  proved for $k$ and prove it for $k+1\leq j$. By applying
  Lemma~\ref{lem:codim1} to $S_1S_{k+1}$ and $S_2S_k$ inside
  $S_2S_{k+1}$, we obtain $S_{k+1}=\langle 1,x\rangle S_k$ meaning
  that $S_{k+1}$ is generated by $(1,x,x^2,\ldots ,x^{k})$.
\end{proof}

As an immediate corollary we obtain
the following theorem, which is proved in \cite{bszVosper} in more
generality. Its proof illustrates the usefulness of
the subspace lattice described above.

\begin{thm}\label{thm:genus0}
  Let $K$ be an algebraically closed field, and let $S$ be a
  finite-dimensional $K$-vector space lying in a field extension $F$
  of $K$. If $\dim S^2=2\dim S-1$ then $S$ has a basis in geometric
  progression, i.e. of the form $(a,ax,ax^2,\ldots ,ax^{n-1})$.
\end{thm}

\begin{proof} Replacing $S$ by $e_1^{-1}S$ reduces to the case
  $e_1=1$. Because the
  codimension of $K$ in $S^2$ is $2\dim S-2$ and is equal to the
  length of any path from $K$ to $S^2$ in the lattice, we have that
  every edge must be of weight $1$. The result therefore follows from
  Lemma~\ref{lem:vosper}.
\end{proof}

We now turn to Riemann-Roch spaces that provide more spaces of low
combinatorial genus~$\gamma$. 


\section{Divisors, Riemann-Roch spaces, and characterising spaces with
combinatorial genus $\gamma=1$}
\label{sec:divisors}
\subsection{Divisors and Riemann-Roch spaces on function fields}
\label{subsec:divisors}
We quickly recall some
basic notions on the theory of the function fields in one variable (or
equivalently of algebraic curves).  For further details, we refer the
reader to \cite{Stichtenoth} or to \cite{Fulton} for a more geometric
point of view.
{In the present subsection, 
in order to introduce some
  notions that will also be useful in the more general setting of
  Section~\ref{sec:perfect}, we do not assume that $K$ is algebraically closed.}

Let $F$ be a function
field over $K$ {such that $K$ is algebraically closed
  in $F$}. 
{
Following \cite[Chapter I]{Stichtenoth}, let us call a {\em place} $P$
of
$F$ the maximal ideal of a valuation ring. Valuations, valuation
rings, places are interchangeable notions in the sense that any one
unambiguously defines the others. Informally, a place captures the concept of ``point'' of the associated
algebraic curve. For a place $P$ we denote by $v_P$ the (unique)
associated discrete valuation. The {\em degree} $\deg P$ of a place $P$ is the
dimension over $K$ of its residue field $\OO/P$. It is always finite
and equal to $1$ when $K$ is algebraically closed.
}

A {\em divisor on} $F$ is
an element of the free abelian group generated by the places of
$F$. Thus it is a formal $\Z$--linear combination of places. Given a
function $f \in F^{\times}$, the divisor of $f$ is defined as
$$
(f) := \sum_{P {\rm \ place\ of\ } F} v_P(f)P.
$$
The
group of divisors is partially ordered as follows: given a divisor
$G = g_1 P_1 + \cdots + g_m P_m$, we have $G \geq 0$ if
$g_1, \dots, g_m \geq 0$. Next $G \geq H$ if $G-H \geq 0$.
{The {\em degree} of the divisor $G$ is defined as
$$\deg G = g_1\deg P_1 + \cdots +g_m\deg P_m.$$
}

Given a divisor $D$ of $F$, the Riemann-Roch space $L(D)$ is defined as
$$
L(D) := \{ f\in F ~|~ (f) + D \geq 0\} \cup \{0\}.
$$
The dimension of this space is given by the famous Riemann-Roch
theorem \cite[Theorem 1.5.15]{Stichtenoth}. In particular it satisfies:
\begin{equation*}
\text{If } \deg D>2g-2 \text{ then }\dim L(D) = \deg(D)+1-g
\end{equation*}
where $g$ denotes the {\em genus} of the field $F$
{(see \cite[Definition 1.4.15]{Stichtenoth} for a
  definition)}.

Two divisors $D$, $D'$ are said to be {\em linearly equivalent} which
we denote by $D \sim D'$ if $D' = D+(f)$ for some function $f \in F^\times$.
Such an equivalence induces an isomorphism between the Riemann-Roch spaces
which is explicit:
$$
\map{L(D')}{L(D)}{s}{fs.}
$$
In short $L(D) = f \cdot L(D')$.

The following well-known result due to Mumford gives an explicit
formula for the product of Riemann-Roch spaces.

\begin{thm}[{\cite[Theorem~6]{Mumford}}]\label{thm:mumford}
  Let $D$, $D'$ be two divisors of a function field $F$ of genus $g$
  {over an algebraically closed field $K$}.
  Suppose that $\deg D \geq 2g$ and $\deg D' \geq 2g+1$.
  Then
  $$
  L(D)L(D') = L(D+D').
  $$
\end{thm}


In particular, combining the Riemann-Roch theorem with Theorem~\ref{thm:mumford}, one
has that in a function field of genus $g$, for any divisor $D$ of degree 
larger than $2g+1$, the space $L(D)$ has combinatorial genus
$\gamma=g$. {This has in particular motivated Definition~\ref{def:gamma}.}

Finally recall that, {over an algebraically closed field $K$}, a function field $F$ has genus~$0$ if and only if
it is a purely transcendental extension $F=K(x)$ of $K$. In such an
extension, the space generated by the functions $1,x,x^2,\ldots ,x^n$ is
equal to the Riemann-Roch space $L(nP_\infty)$ where $P_\infty$ is the
place at infinity.
We remark that the statement of Theorem~\ref{thm:genus0} is
equivalent to saying that a space $S$ has combinatorial genus $\gamma=0$ if
and only if it has genus $0$ and is equal to a Riemann-Roch space
$L(D)=aL(nP_\infty)$ for a divisor $D=nP_\infty-(a)$ and a function~$a$.

\subsection{Statement of the main theorem}
\label{sec:Vosper+1}
 { Until the end of Section~\ref{sec:genus0}, the base field
 $K$ is supposed to be algebraically closed.}

 Our main purpose is to
classify spaces over $K$ with combinatorial genus $1$.  That is to say, given a
finitely generated field $F$ over $K$ we want to understand the structure
of $K$--spaces $S \subseteq F$ such that $\dim S^2 = 2 \dim S$.
As remarked at the end of Section~\ref{sec:freiman}, we know from
Theorem~\ref{thm:freiman_field} that the
transcendence degree of $S$ must be equal to $1$, so that, assuming without
loss of generality $1\in S$, $F$ is a function field (in one
variable). 

With the case $\gamma=1$ in mind, we recall for future reference
its additive analogue:
\begin{prop}\label{prop:A+A=2A}
  A subset $A$ of the integers, $|A|\geq 4$, is such that $|A+A|=2|A|$ if and only
  if it is of the form $A=a+\{0,2d,3d,\ldots ,(n-1)d,nd\}$ for some
  {integer $a$ and some}
  non-zero $d$: in other words, writing $A$ as an increasing sequence, 
  it is an arithmetic
  progression with a missing element after the first position
  (positive $d$) or before the last position (negative $d$).
\end{prop}
Indeed, Freiman's $3k-4$ Theorem applied to the case $|A+A|=2|A|$
gives that $A$ is an arithmetic progression with a single missing
element, which is then easily seen to be necessarily at an extreme end
of the progression. 
{
Proposition~\ref{prop:A+A=2A} is also true in the
integers modulo a prime $p$, provided $|A+A|\leq p-2$, see \cite{hr00}
(and is therefore necessarily true in $\Z$).
}

To generate a space $S$ such that $\gamma=1$, Proposition~\ref{prop:A+A=2A}
suggests naturally to take a basis of the form $x^a, a\in A$, where
$A$ is such that $|A+A|=2|A|$. Such a space is a subspace of
codimension $1$ inside a space with a basis in geometric progression,
i.e. inside a Riemann-Roch space of genus $0$.

Alternatively, Theorem~\ref{thm:mumford} tells
us that Riemann-Roch spaces of genus $1$ will also give us
spaces with combinatorial genus equal to
$1$.

Our main result states, broadly speaking, that the two
constructions above cover all possible cases:

\begin{thm}\label{thm:main}
  { Let $K$ be an algebraically closed field
  and $F$ be a function field over $K$.}
  Let $S \subseteq F$, $1\in S$, be a space of finite dimension
  $n \geq 4$ and combinatorial genus $\gamma=1$. Then $S$ has genus
  $0$ or $1$. Moreover,
  \begin{itemize}
    \item if $S$ has genus $1$  then $S=L(D)$ for $D$ a divisor of degree $n$,
    \item if $S$ has genus $0$, then $S$ is a subspace of codimension
      $1$ inside a space $L(D)$ for $D$ a divisor of degree $n$.
  \end{itemize}
 \end{thm}

We remark that in the genus $g=0$ case, all subspaces of codimension
$1$ inside an $L(D)$
space do not necessarily have combinatorial genus $\gamma$ equal to
$1$. 
{ For instance the space $S = \langle 1, x, x^3, x^4 \rangle$
has codimension $1$ in a Riemann Roch space of a field of genus
$0$, while it has combinatorial genus $2$.}
We postpone to Section~\ref{sec:genus0} the precise
characterisation of such subspaces which is slightly more involved
than in the additive case given by Proposition~\ref{prop:A+A=2A}.

\section{Proof of Theorem~\ref{thm:main}}\label{sec:proof_main}

\subsection{Overview}
Since the proof of the theorem is somewhat lengthy, we give an
outline.

As we have argued before, we may always assume that $1\in S$, and that
$F=K(S)$. We start by fixing an arbitrary place $P$ and an associated
$P$-filtered basis which, possibly after replacing $S$ by a
multiplicative translate $s^{-1}S$, is of the form
$$e_1=1,e_2=x,e_3=y,\ldots, e_n$$
with valuations in decreasing order $v_P(1)=0>v_P(x)>v_P(y)>\cdots
>v_P(e_n)$.

Next we define (Section~\ref{sec:divisor}) the divisor $D_U$ of a
space $U$ to be the smallest {divisor such that the}
Riemann-Roch space {$L(D_U)$} contains $U$. Our
strategy will be to study closely the chain of divisors $D_{S_i}$ for
the filtration $(S_i)$ of $S$ relative to $P$. Our goal will be to
show that $S_i$ is either equal to $L(D_{S_i})$ or of codimension $1$
inside $L(D_{S_i})$, for all $i\geq 2$.

We will consider closely the lattice of subspaces introduced in
Section~\ref{sec:lattice} and exploit the fact that most of its edges are of
weight $1$. A crucial intermediate result will be Lemma~\ref{lem:D_Si}
which will tell us that the divisor increments $D_{S_{i+1}}-D_{S_i}$ 
must all be equal to $D_{S_3}-D_{S_2}$, except possibly for one index $i$, that we
call the $P$-index of $S$, which is the
unique index $i$ for which $\dim S_iS_{i+1}>\dim S_i^2+1$.
Sections 
\ref{sec:separation} and \ref{sec:Pindex} build up material leading up
to Section~\ref{sec:changing} which derives Lemma~\ref{lem:D_Si}.

Section~\ref{sec:genus} considers next the algebraic equations
satisfied by $x$ and $y$. We will show that $F=K(S)=K(x,y)$ and that
$F$ must be of genus $g\leq 1$. We then turn to determining the
sequence $(D_{S_i})$ exactly. Lemma~\ref{lem:D_Si} tells us that when
the $P$-index equals $2$, the divisors $D_{S_2}$ and $D_{S_3}$
determine the whole sequence. In Section~\ref{sec:Pindex=2} we show
that in this case we must have $D_{S_2}=P+Q$ and $D_{S_3}=2P+Q$ for
some place $Q$ (possibly equal to $P$), so that the
whole sequence of divisors must take the form
$0,P+Q,2P+Q,3P+Q,\ldots ,(n-1)P+Q$.  Section~\ref{sec:Pindex>2} deals
with the remaining case, for which it is shown that the sequence of
divisors must be of the form $0,P,2P,\ldots ,(n-2)P,(n-1)P+Q$.

\subsection{Minimum valuations, the divisor of a space}
\label{sec:divisor}

The following lemma is straightforward:

\begin{lemma}\label{lem:sum_min_val}
  Let $U,V$
  be two finite dimensional $K$--spaces in $F$. Let $v$ be a valuation on $F$.
  Then,
  $\min v(UV) = \min v(U)+ \min v(V)$.
\end{lemma}

Note that $\min v(U) = 0$ for all valuations $v$ on $F$ but finitely many
of them. This justifies the following:

\begin{defn}\label{def:DU}
We denote by $D_U$ the divisor
$$
D_U := \sum_{P,\ {\rm place \ of\  }F} -\min v_P(U)P.
$$
This is the smallest divisor $D$ such that $U \subseteq L(D)$.  
\end{defn}

\subsection{Separation}
\label{sec:separation}

\begin{defn}[Separation]
  Given a finite dimensional $K$-space
  $U \subset F$ and two distinct places $P_1$ and $P_2$, one says that
  $U$ {\em separates} $P_1$ and $P_2$ if
  there exists $f_1, f_2 \in U$ such that
  \begin{enumerate}
    \item $v_{P_i} (f_i) = \min v_{P_i}(U)$ for $i \in \{1,2\}$;
    \item for $i \neq j \in \{1,2\}$, $v_{P_i}(f_j) > \min v_{P_i}(U)$.
  \end{enumerate}
\end{defn}

Given a $K$-space $U \subset F$ and a place $P$, let $U_P := \{
u \in U ~|~ v_P(u) > \min v_P(U)\}$. Such a space has codimension $1$
in $U$ and the notion of separation of two distinct places can be reformulated
as follows.

\begin{lemma}\label{lem:U_pU_Q}
  The space $U$ separates two distinct places $P, Q$ if and only if
  $U_P \neq U_Q$.
\end{lemma}

A first example of spaces having a good property of separation are
Riemann-Roch spaces. The following statement is classical, we provide
a proof for the sake of self-containedness.

\begin{lemma}\label{lem:L(D)sep}
  Let $F$ be a function field of genus $g$ and let $D$ be a divisor of $F$
  such that $\deg D > 2g$. Then the space $L(D)$
  separates any two places $P$ and $Q$. 
\end{lemma}

\begin{proof}
  Since $\deg (D) > 2g$, we have as a consequence of the Riemann-Roch
  Theorem
  \begin{eqnarray*}
    \dim L(D-Q) &=& \dim L(D)-1\\
    \dim L(D-P) &=& \dim L(D)-1\\
    \dim L(D-P-Q) &=& \dim L(D)-2.
  \end{eqnarray*}
  Indeed, all the considered divisors have degree greater than $2g-2$.
  If we set $U = L(D)$ we obtain $U_P \neq U_Q$ and conclude using
  Lemma~\ref{lem:U_pU_Q}.
\end{proof}

The next lemma deals with separation in products of spaces.

\begin{lemma}\label{lem:sep_products}
Let $U, V \subseteq F$ be two $K$-spaces and $P, Q$ two places of $F$.
Then, $UV$ separates $P$ and $Q$ if and only if $U$ or $V$ separates
$P$ and $Q$. 
\end{lemma}

\begin{proof} Let us suppose first that $U$ separates $P$ and $Q$. 
Let $a \in U$ be such that $v_P(a) = \min v_P(U)$ and $v_Q(a) > \min v_Q(U)$.
Let $b \in V$ be such that $v_P(b) = \min v_P(V)$.
We have $ab \in (UV)_Q \setminus (UV)_P$, hence $UV$ separates $P$
and $Q$. 

Conversely, suppose that neither $U$ nor $V$ separates $P$ and $Q$. 
Let $u \in U$ and $v\in V$ be such that 
$$
U = U_P \oplus Ku \qquad {\rm and} \qquad V = V_P \oplus Kv.
$$
Then, 
$$
UV = (U_P V_P + uV_P + U_P v) + Kuv. 
$$
Clearly $(U_P V_P + uV_P + U_P v) \subseteq {(UV)}_P$ and since
$(U_P V_P + uV_P + U_P v)$ has codimension at most $1$ in $UV$ we conclude
that
$$
(U_P V_P + uV_P + U_P v) = {(UV)}_P.
$$
By assumption, we have $U_P = U_Q$ and $V_P = V_Q$ and hence
$$
{(UV)}_P = U_Q V_Q + u V_Q + U_Q v = {(UV)}_Q
$$
so $UV$ does not separate $P$ and $Q$ either. This concludes the proof.
\end{proof}

\subsection{The lattice of subspaces and the $P$-index of a space}
\label{sec:Pindex}

  For the remainder of Section~\ref{sec:proof_main}, $P$ is a fixed arbitrary place of
  $F=K(S)$. We choose a filtered basis
  $(e_1, \ldots, e_n)$ where, having replaced $S$ by $e_1^{-1}S$ if necessary,
  we have set $e_1=1$. We consider the 
  filtration $S_1 \subset S_2\subset \cdots \subset S_n = S$
  associated to $P$, together with the lattice of subspaces $S_iS_j$ introduced in
  Section~\ref{sec:lattice} and illustrated in
  Figure~\ref{fig:subspaces}.
  Recall that the weight of an edge $V\rightarrow W$ is given by the
  codimension of $V$ in $W$. In the case $\dim S^2=2\dim S$ we have:

\begin{lemma}\label{lem:weight2}
  All edges lying on a directed path from $S_1^2$ to $S_n^2 = S^2$
  have weight $1$ except for an edge which has weight $2$.
\end{lemma}

\begin{proof}
  The path has $2n-2$ edges, while $\dim S_1^2 = 1$ and $\dim S_n^2
  =\dim S^2 = 2n$, therefore the codimension of $S_1^2$ in $S_n^2$,
  which is also the sum of weights on the path, equals
  $2n-1$. Remembering that weights are at least $1$, the result follows.
\end{proof}

\begin{lemma}\label{lem:vertical}
  In the subspace lattice, every vertical edge from $S_iS_j$ to
  $S_{i+1}S_j$ has weight~$1$ for $i\geq 2$. 
\end{lemma}

\begin{proof}
  If not, then such an edge has weight $2$, which implies that every
  edge on the sublattice of directed paths from $S_1^2$ to $S_iS_j$ has
  weight $1$ by Lemma~\ref{lem:weight2}. But then,
  Lemma~\ref{lem:vosper} implies that $S_j$ has a basis in geometric
  progression, which in turn implies that $\dim S_j^2=2\dim S_j-1$,
  meaning that all edges on the sublattice of paths from $S_1^2$ to
  $S_j^2$ have weight $1$, a contradiction.
\end{proof}

\begin{figure}[h]
    \centering
    \begin{tikzpicture}[scale=0.8]
\node[name=S22] at (4,4) {$S_2^2$};
\node[name=S23] at (6,4) {$S_2S_3$};
\node[name=S24] at (8,4) {$S_2S_4$};
\node[name=S25] at (10,4) {$S_2S_5$};
\node[name=S26] at (12,4) {$S_2S_6$};
\node[name=S33] at (6,6) {$S_3^2$};
\node[name=S34] at (8,6) {$S_3S_4$};
\node[name=S35] at (10,6) {$S_3S_5$};
\node[name=S36] at (12,6) {$S_3S_6$};
\node[name=S44] at (8,8) {$S_4^2$};
\node[name=S45] at (10,8) {$S_4S_5$};
\node[name=S46] at (12,8) {$S_4S_6$};
\node[name=S55] at (10,10) {$S_5^2$};
\node[name=S56] at (12,10) {$S_5S_6$};
\node[name=S66] at (12,12) {$S_6^2$};

\draw (S22) --node[above]{$1$}  (S23)--node[above]{$1$}(S24)--node[above]{$2$}(S25)--node[above]{$1$}(S26)--node[left]{$1$}(S36)--node[left]{$1$}(S46)--node[left]{$1$}(S56)--node[left]{$1$}(S66);
\draw (S33)--node[above]{$1$}(S34)--node[above]{$2$}(S35)--node[above]{$1$}(S36);
\draw (S44)--node[above]{$2$}(S45)--node[above]{$1$}(S46);
\draw (S25)--node[left]{$1$}(S35)--node[left]{$1$}(S45)--node[left]{$1$}(S55)--node[above]{$1$}(S56);
\draw (S23)--node[left]{$1$}(S33);
\draw (S24)--node[left]{$1$}(S34)--node[left]{$1$}(S44);
\end{tikzpicture}
    \caption{$n=6$ and $j_0=4$ for the $P$-index}
    \label{fig:Pindex}
  \end{figure}
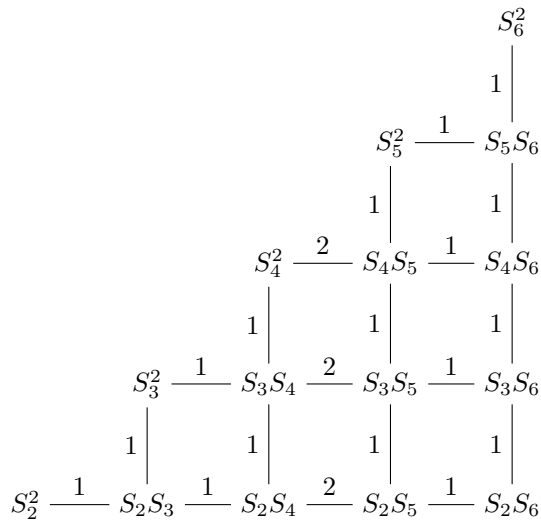

Since any path from $S_1^2$ to $S_n^2$ has exactly one edge of weight
$2$ (Lemma~\ref{lem:weight2}), Lemma~\ref{lem:vertical}
implies that all horizontal edges $S_iS_j \rightarrow S_iS_{j+1}$ of
weight $2$ occur for a common index $j$, that we call the {\em
  $P$--index} of the space $S$.
Summarising:

\begin{lemma}\label{lem:Pindex}
  There is an index $j_0$, called the {\em $P$-index} of $S$, such that
  $\dim(S_{i}S_{j_0+1})-\dim(S_iS_{j_0})=2$
  for every $i=2\ldots j_0$.
\end{lemma}

The above statement is illustrated on Figure~\ref{fig:Pindex}.

We will see later that there are only two possible values for the
$P$-index, namely $j_0=2$ and $j_0=n-1$.

\subsection{Changing the valuation}
\label{sec:changing}

A single valuation $v_P$ may not be enough to describe sufficiently the
spaces $S_i$, and it will be useful to involve alternative valuations.
We now argue that some information
obtained from a valuation $v_P$ may be ``transferred'' and hence provide 
some information with respect to another valuation $v_Q$.
When all weights are equal to $1$ on the sublattice from $S_iS_j$ to
$S_{i+1}S_{j+1}$, i.e. in the situation illustrated on Figure~\ref{fig:square},
\begin{figure}[h]
  \centering
  \begin{tikzpicture}
\node[name=S00] at (0,0) {$S_iS_j$};
\node[name=S01] at (0,2) {$S_{i+1}S_{j}$};
\node[name=S10] at (2,0) {$S_{i}S_{j+1}$};
\node[name=S11] at (2,2) {$S_{i+1}S_{j+1}$};

\draw (S00) --node[above]{$1$} (S10) --node[left]{$1$} (S11)
--node[above]{$1$} (S01)--node[left]{$1$} (S00);
\end{tikzpicture}
  \caption{weights are equal to $1$}
  \label{fig:square}
\end{figure}
we have already observed (Lemma~\ref{lem:codim1}) that
$S_{i}S_{j+1}=S_{i+1}S_{j}$. 

>From this equality, we conclude that for any place $Q$,
we have 
$$v_Q(S_{i}S_{j+1})=v_Q(S_{i+1}S_{j}),$$ 
and hence,
thanks to Lemma~\ref{lem:sum_min_val},
\begin{equation}\label{eq:delta}
-\min v_Q(S_{j+1})+\min v_Q(S_j) =-\min v_Q(S_{i+1})+\min v_Q(S_i).
\end{equation}
Applying \eqref{eq:delta} when $i=2$ yields the following useful lemma.
\begin{lemma}\label{lem:D_Si}
For every $j\geq 2$ that differs from the $P$-index $j_0$ of $S$,
we have $D_{S_{j+1}}-D_{S_j}=D_{S_3}-D_{S_2}$.
\end{lemma}

We now turn to determining the genus of $F$.

\subsection{The genus of the ambient field}\label{sec:genus}

We prove first that $F=K(S)$ is in fact generated by the subspace $S_3$.

  \begin{lemma}\label{lem:dimSiSi+1}
    We have $F = K(S_3)$.
  \end{lemma}
  \begin{proof}
    For any
    $i \geq 1$, we have
    $$
    S_{i-1} S_i  = S_{i-1}^2 +  S_{i-1} e_i.
    $$
    Lemma~\ref{lem:weight2} asserts
    $\dim (S_iS_{i-1}) - \dim S_{i-1}^2 \leq 2$. Moreover, if $i \geq 4$,
    $\dim S_{i-1} \geq 3$ and hence the intersection
    $S_{i-1}^2 \cap S_{i-1}e_i$ is non-zero.
    Consequently, $e_i \in K(S_{i-1})$. Therefore 
    $F=K(S_n)=K(S_{n-1})=\cdots =K(S_3)$.
  \end{proof}

Remembering that
$e_1=1$, and that we use the notation $x=e_2$ and $y=e_3$, 
Lemma~\ref{lem:dimSiSi+1} says that $F=K(x,y)$. Our next goal
will be to determine the genus of $F$ and for this, we will identify an equation of lowest degree satisfied by $x$ and $y$: its
degree will determine the genus of $F$. 
  
\begin{prop}\label{prop:genus}
  The field $F$ has genus less than or equal to $1$.
\end{prop}

\begin{proof}
  Depending on the value of the $P$--index being greater than $2$ or
  equal to $2$, we have either $\dim S_2 S_3 = 4$ or $\dim S_2 S_3 = 5$.
Moreover, $S_2S_3$ is generated by $(1,x,x^2,xy,y)$. 

If $\dim S_2 S_3 = 4$ then we get a linear relation between
  $1,x,x^2,xy,y$ which immediately shows that $y\in K(x)$ and
  consequently  that $F=K(x)$ and has genus $0$. We note that we have
  found an irreducible  quadratic relation between $x$ and $y$,
  meaning that $F$  is the function field of a
      plane irreducible conic. 

If   $\dim S_2 S_3 = 5$, then $1,x,x^2,xy,y$ are linearly
  independent over $K$; the subspace $S_3^2$, which is generated by
  $(1, x, y, x^2, x y, y^2)$ and has dimension $6$, does also not produce
  an algebraic relation between $x$ and $y$. We need to go to
  $S_3S_4$, which is of dimension $7$ and is generated by $(1, x, y, e_4, x^2, x y, x e_4,
      y^2, y e_4)$. It  entails the existence of two
      independent relations
      \begin{eqnarray}
       && e_4 L_1(x,y) = Q_1(x, y) \label{eq:relation1}\\
       && e_4 L_2(x,y) = Q_2(x, y) \label{eq:relation2}
      \end{eqnarray}
      where $L_1, L_2$ are linear polynomials and $Q_1, Q_2$ are
      quadratic polynomials.
      Moreover, the linear polynomials $L_1, L_2$ are nonzero since 
      $\dim S_3^2 = 6$ and hence, there is no quadratic polynomial vanishing
      on $x, y$.  By eliminating  $e_4$ we get
      \begin{equation}
      L_1(x, y) Q_2 (x, y) = L_2(x, y) Q_1(x, y).
      \end{equation}
      The polynomial
      $L_1 Q_2 - L_2 Q_1$ is nonzero because the relations~(\ref{eq:relation1})
      and~(\ref{eq:relation2}) are independent.  This polynomial has degree at
      most $3$ and, since $\dim S_3^2 = 6$, there is no quadratic
      relation relating $e_1, x, y$, which asserts that the degree
      is exactly $3$. Therefore, the genus of $F$ is either $0$ (if
      the curve of equation $L_1 Q_2 - L_2 Q_1$ has a singularity) or
      $1$, as a consequence of B\'ezout's Theorem \cite{Fulton}.
\end{proof}

\noindent {\bf Summary.} Writing $S_2 = \langle 1,x\rangle$ and
$S_3 = \langle 1,x,y \rangle$, we may distinguish three cases.
\begin{enumerate}
  \item\label{item:conic} $\dim S_2S_3 = 4$. In this situation,
    $F=K(x)$ and 
    there is an irreducible quadratic polynomial $Q$ such that $Q(x,y) = 0$.
  \item\label{item:sing_cubic} $\dim S_2S_3 = 5$ and there is a cubic
    relation $$L_1(x,y) Q_2(x,y) - L_2(x,y) Q_1(x,y) = 0$$ such that the
    corresponding projective plane curve is singular.
  \item\label{item:smooth_cubic} $\dim S_2S_3 = 5$ and there is a cubic
    relation $$L_1(x,y) Q_2(x,y) - L_2(x,y) Q_1(x,y) = 0$$ such that the
    corresponding projective plane curve is smooth.
\end{enumerate}
Cases (\ref{item:conic}) and (\ref{item:sing_cubic}) correspond to
the genus $0$ case. Case (\ref{item:smooth_cubic}) correspond to the
genus $1$ case. 

\smallskip

In order to finish the proof of Theorem \ref{thm:main}, it remains to
compute the divisor $D_S$, as defined in  \ref{def:DU}. For this, we
will first determine $D_{S_2}$ and $D_{S_3}$, and then, iteratively
compute  $D_{S_i}$ for $i\leq n$. 

\subsection{The $P$-index is equal to $2$}
\label{sec:Pindex=2}

In this subsection, we treat the case of the $P$-index being equal to
$2$, which amounts to $\dim S_2S_3=5$. We have already proved that $x$
and $y$ satisfy an equation of degree $3$; in the next lemma we show that
moreover this equation has a specific form.

\begin{lemma}\label{lem:deg2}
  The field $F = K(x,y)$ has degree $2$ over $K(x)$ and the equation
  satisfied by $x$ and $y$ is of the form 
\begin{equation}\label{eq:G}
y^2+B(x)y+C(x)=0
\end{equation}
where $\deg B\leq 2$ and $\deg C\leq 3$.
\end{lemma}

\begin{proof}
  The field $F$ is generated by $x$ and $y$ and the proof of
  Proposition~\ref{prop:genus} has shown that there is a relation $G(x,y) = 0$
  of degree $3$. Suppose that $G(x,y)$ contains a term in $y^3$ and write $G(x,y)$ as
  \begin{equation*}
  G(x, y) = y^3 + A(x)y^2 + B(x)y + C(x) = 0
  \end{equation*}
 where $\deg A \leq 1$, $\deg B \leq 2$ and $\deg C \leq 3$. 
By construction, we have $v_P(y^3)<v_P(xy^2)<v_P(x^2y)<v_P(x^3)$ so
$v_P(y^3)<v_P(A(x)y^2+B(x)y+C(x))$ which is in contradiction with
$G(x,y)=0$. So the equation has the form $A(x)y^2 + B(x)y + C(x) =
0$. With a similar reasoning we can see that moreover $A(x)$
must be a constant.

It remains to rule out the case when $A(x)=0$,
which would mean that $y\in K(x)$ and  $F=K(x)$. Let us assume we are
in this case and reach a contradiction. Because $v_P(x)<0$,
$P=P_\infty$ is the place at infinity of $K(x)$ and $v_P(x)=-1$. In
particular $D_{S_2}=L(P)$. Regarding $y$, we know that $v_P(y)<v_P(x)$
so the only possibility is $v_P(y)=-2$ because of the structure of
$v_P(S)$ which contains $\{0,-1\}$: indeed, recall from Proposition~\ref{prop:A+A=2A} that it  can have a
missing element only after its first or before its last value and
since $n\geq 4$ the value before the smallest value of $v_P(S)$ cannot be 
$v_P(y)$. According to
Lemma \ref{lem:L(D)sep}, $S_2$ separates $P$ and any other place
$Q\neq P$; according
to Lemma \ref{lem:sep_products}, so does $S_2S_4$. But $S_2S_4=S_3^2$ from Lemma
\ref{lem:codim1} which entails that also $S_3$ separates $P$ and $Q$. So $y$
cannot have a pole at $Q$, which leaves the only possibility
$y=D(x)$ {for some polynomial $D$ with} $\deg(D)=2$. But this situation  is not compatible with the
condition that $\dim S_2S_3=5$.
\end{proof}

\begin{prop}
  When $\dim S_2S_3=5$, there is a place $Q$, possibly equal to $P$,
  such that, for all $i=2,\ldots,n$, 
$$D_{S_i}=(i-1)P+Q.$$
\end{prop}

\begin{proof}
  We first focus on determining the divisors of $S_2$ and $S_3$.
  Because $v_P(x)<0$, the place $P$ is above the place at infinity of
  $K(x)$, that we will denote $P_\infty$. Since
  $K(x,y)$ has degree $2$ over $K(x)$ by Lemma~\ref{lem:deg2}, we have from
  \cite[Ch.3]{Stichtenoth} that
  $P_{\infty}$ decomposes in $F=K(x,y)$ either as $2P$ (the ramified
  case) or as $P+Q$ where $Q\neq P$ (the split case). For any other
  place $R\notin\{P, Q\}$, the valuation $v_R(x)$ is non negative, and
  we have therefore $D_{S_2}=2P$ in the ramified case and
  $D_{S_2}=P+Q$ in the split case. We now focus on determining
  $D_{S_3}$.  From \eqref{eq:G} the valuation $v_R(y)$ can only be non
  negative for any place $R \notin\{P, Q\}$, so we are
  left with determining the valuation at $P$ and $Q$ of $y$.

We now view $K(x,y)$ as an algebraic extension of $K(y)$ instead of an
extension of $K(x)$. We remark that $P$, respectively $P$ and $Q$ in the split case, are also the places above the place at infinity of $K(y)$. So, since $[K(x,y):K(y)]\leq 3$,  we also
know from \cite[Ch. 3]{Stichtenoth}  that $v_P(y)\geq -3$, and, in the split case, that $v_P(y)+v_Q(y)\geq -3$. Taking account of this, we see that in the ramified case $2P$ we necessarily have $v_P(x)=-2$ and
$v_P(y)=-3$, and so, $D_{S_2}=2P$ and $D_{S_3}=3P$.

In the split case $P+Q$, i.e. $v_P(x)=v_Q(x)=-1$, we can conclude so far that $v_P(y)=-2$ or $-3$. But the case $v_P(y)=-3$ would create
a forbidden hole in $v_P(S)$ that contains $\{v_P(1)=0,v_P(x)=-1\}$
(because not in first or last position). So the only possibility is $v_P(y)=-2$. 
Now since $v_P(y)+v_Q(y)\geq -3$ we must have $v_Q(y)\geq -1$. So
$D_{S_2}=P+Q$ and $D_{S_3}=2P+Q$.

Finally, to obtain $D_{S_i}=(i-1)P+Q$ for $i\geq 4$ we apply Lemma~\ref{lem:D_Si}.
\end{proof}

In the case when $\dim S_2S_3=5$, the above proposition concludes the proof of
Theorem \ref{thm:main}. Indeed, in the case when $F$ is of genus $1$,  Riemann-Roch theorem tells
us that $S$ must coincide with the space $L((n-1)P+Q)$. In the genus
$0$ case, $S$ is of codimension $1$ inside $L((n-1)P+Q)$.

\subsection{The $P$-index is greater than $2$, or the plane conic
  case}
\label{sec:Pindex>2}
If $S_2 S_3$ has dimension $4$, then recall that $F=K(x)$ and that Lemma~\ref{lem:vosper} implies
that $S_i$ is generated by $(1,x,\dots,x^{i-1})$ for every $i$, $2\leq i\leq k$, where $k$ is
the $P$-index of $S$; in other words, in this range, $S_i =L((i-1)P)$
where $P=P_{\infty}$ is the place at infinity of $K(x)$.

\begin{lemma}\label{lem:kP+Q}
  We have $D_{S_{k+1}} = kP+Q$
for some place $Q$ possibly equal to $P$.
\end{lemma}

\begin{proof}
 Let $z$ be such that $S_{k+1}=S_k+Kz$ with $z$ of minimum
  $P$-valuation in $S_{k+1}$. We already know that $v_P(S_{k+1})$ is
  either an arithmetic progression or an arithmetic progression with a
  missing element, in other words either $v_P(z)=-k$ or
  $v_P(z)=-(k+1)$.

  Consider first the case $v_P(z)=-(k+1)$. Consider the product
  $S_3S_{k+1}$ which must be of dimension $(k+1)+3$. The space $S_3$
  is generated by $1,x,x^2$ and we have
  $v_P(S_3S_{k+1})\supset\{v_P(x^k), v_P(z), v_P(xz),v_P(x^2z)\}$ so that 
  $$v_P(S_3S_{k+1})=\{0,-1,-2,\ldots,-(k-1),-k,-(k+1),-(k+2),-(k+3)\}.$$
Since $x^k=xx^{k-1}$ and $x^{k+1}=x^2x^{k-1}$ are contained in $S_3S_{k+1}$,
we have that $S_3S_{k+1}$ contains the subspace generated by the
geometric progression $1,x,x^2,\ldots ,x^{k+1}$, which is equal to
the subspace of $S_3S_{k+1}$ of functions of $P$-valuation $\geq
-(k+1)$, because this must be a space of dimension $k+2=\dim
S_3S_{k+1}-2$. 
Since $v_P(z)=-(k+1)$, the function $z$ must belong to the
aforementioned subspace, meaning that $z$
is a polynomial in $x$ of degree $k+1$,
in other words $S_{k+1}\subset L((k+1)P)$.

Consider now the remaining case $v_P(z)=-k$. The set $v_P(S_{k+1})$ is
now the arithmetic progression $\{0,-1,\ldots, -k\}$ and we have
$D_{S_{k+1}}= kP+D$ for some positive divisor $D$. 
Consider again the product $S_3S_{k+1}$ and let $U$
be the subspace of those elements of $S_3S_{k+1}$ that have a
valuation { at $P$}
greater than the minimum, namely $-k-2=v_P(z)+v_P(x^2)$. 
We have that $U$ has codimension $1$
in $S_3S_{k+1}$ (Proposition~\ref{prop:filtered_basis}) i.e., $\dim
U=k+3$. Note also that $U$ contains 
$1,x,\ldots,x^{k+1}$, so that there exists $u\in U$ of positive
$P$-valuation, such that $u,1,x,\ldots,x^{k+1}$ is a basis of~$U$.
Now since $v_P(z)=-k$, we have $z\in U$ and 
\begin{equation}
  \label{eq:pole}
  z = P_{k}(x)+\lambda u, \quad \lambda\in K
\end{equation}
with $P_{k}(x)$ a polynomial in $x$ of degree at most $k$.
Note that we must have $\lambda~\neq~0$ otherwise, since $v_P(z)=-k$,
$z$ is a polynomial of degree $k$ in $x$ contradicting that 
$\dim S_3S_{k+1}=k+4$. The equality \eqref{eq:pole} implies therefore that
$\alpha\in K$ is a pole of $z$ if and only if it is a pole of $u$.
For such a pole $\alpha$, we have $(x-\alpha)z\in U$ since
$v_P((x-\alpha)z)=-(k+1)$, hence
 $$(x-\alpha)z = Q_{k+1}(x) + \mu u, \quad \mu\in K$$
for $Q_{k+1}(x)$ a polynomial in $x$ of degree at most $k+1$.
This implies that $\mu=0$ otherwise the left hand side and the right
hand side would not have the same $\alpha$-valuation. This
proves that $z$ has a pole of order $1$ at $\alpha$ and
simultaneously that $z$ cannot have a pole at $\beta$ for $\beta\neq \alpha$.
Therefore $z$ has a single pole of order $1$ besides $P_\infty$.
\end{proof}

We conclude with the following statement.

\begin{prop}
  The $P$--index $k$ of $S$ equals $n-1$.
\end{prop}

\begin{proof}
  In the case when $Q=P$ in Lemma~\ref{lem:kP+Q}, since the set of
  $P$-valuations can only be an arithmetic progression with a hole in
  the last position, we must have $k=n-1$. We may therefore suppose
  that $Q\neq P$. 

  Suppose towards a contradiction that $k \leq n-2$.
  The space $S_k$ is a Riemann-Roch space. Hence,
  from Lemma~\ref{lem:L(D)sep}, $S_k$ separates $P$ with any place
  $Q\neq P$ 
  of $F$. Therefore, from Lemma~\ref{lem:sep_products} 
  so does $S_kS_{k+2}$. On the other hand $S_{k+1}$ does not separate
  $P$ and $Q$ and hence, again applying Lemma~\ref{lem:sep_products}, 
  $S_{k+1}^2$ does not separate them either. 
  This is a contradiction since, from Lemma~\ref{lem:D_Si}, 
  $S_{k}S_{k+2}$ should be equal to $S_{k+1}^2$.
\end{proof}

As a conclusion, in this situation, $S$ is a subspace of 
codimension $1$ of a Riemann-Roch space of the form $L((n-1)P+Q)$
where $Q$ is a place, possibly equal to $P$.

\section{Further description of spaces with genus $0$ and
  combinatorial genus $1$}\label{sec:genus0}

Theorem~\ref{thm:main} gives a complete characterisation of spaces $S$ of
genus $1$ with combinatorial genus $\gamma=1$ by saying that they are
exactly Riemann-Roch spaces. However, in the case when the genus of the
field $F$ is $0$, it only says that $\gamma=1$ implies that $S$ is of
codimension~$1$ inside a Riemann-Roch space: but not all subspaces of
codimension $1$ inside an $L(D)$ space have combinatorial genus $1$,
so this raises the question of exactly which subspaces have
$\gamma=1$. The following theorem gives a precise answer.

\begin{thm}\label{thm:g=0}
  Let $S$ be of genus $g=0$ and of combinatorial genus $\gamma=1$. Then, up to multiplication by a constant, 
$S$ has a basis
of one of the following two types:
\begin{itemize}
\item[(i)]
$1,t,t^2,\ldots ,t^{n-2},(t+\alpha)t^{n-1}$, 
\item[(ii)]
$1,(t+\alpha)t,(t+\alpha)t^2,\ldots , (t+\alpha)t^{n-2}, (t+\alpha)t^{n-1}$
\end{itemize}
for some function $t$ and some constant $\alpha\in K$.
\end{thm}

Before proving Theorem~\ref{thm:g=0} we introduce an intermediate
result. The proof of Theorem~\ref{thm:main} has shown that $S$
(after replacing it by a suitable multiplicative translate $s^{-1}S$) 
is such that $1\in S$ and $S\subset L((n-1)P+Q)$ where $P$ is the
initial arbitrary place of $F=K(S)$ and $Q$ is some place that may or may
not be equal to $P$. The following proposition states that there
always is a choice of $P$ for which we have $Q=P$.

\begin{prop}\label{prop:goodP}
  There exists a place $P$ and a function $s\in S$ such that
  $s^{-1}S\subset L(nP)$.
\end{prop}

\begin{proof}
  We start with an arbitrary choice of $P$ so that we may suppose
   $1\in S$ and $S\subset L((n-1)P+Q)$ with $Q\neq P$. Since the
   action of $\mathrm{PGL}(2,K)$ on places is $3$-transitive, we may
   choose a function $t$ for which $F=K(t)$ and such that $P$ and $Q$,
   viewed over $K(t)$, are the place at infinity and the place at zero
   respectively. In other words $L((n-1)P+Q)$ is the space of Laurent
   polynomials of the form
   \begin{equation}
     \label{eq:laurent}
     f(t)=\frac {a_{-1}} t + a_0 + a_1  t + a_2 t^2 + \cdots + a_{n-1}
   t^{n-1}.
   \end{equation}
   Since $S$ has codimension $1$ inside $L((n-1)P+Q)$, there exist
   coefficients $\lambda_{-1},\lambda_0,\ldots,\lambda_{n-1}$ in $K$,
   such that $S$ consists of the space of functions \eqref{eq:laurent}
   satisfying 
   \begin{equation}
     \label{eq:lambda}
     \lambda_{-1}a_{-1}+\lambda_0a_0+\cdots + \lambda_{n-1}a_{n-1} = 0.
   \end{equation}
    If $\lambda_{-1}=0$ then $\frac{1}{t}\in S$ so that $1\in tS$ and
   $tS\subset L(nP)$ and we are finished.
   Suppose therefore $\lambda_{-1}\neq 0$. We claim there exists $a\in
   K$ such that the function $(t-a)^n\in tS$. Indeed, expanding the
   expression $(t-a)^n$ as
   $$(t-a)^n=a_{-1}+a_0t+\cdots +a_{n-1}t^n$$
   we see that the quantity $\lambda_{-1}a_{-1}+\cdots +
   \lambda_{n-1}a_{n-1}$ is a polynomial in $a$ of degree exactly $n$,
   which has roots in $K$
   since $K$ is algebraically closed. For such an $a$ we get that
   \eqref{eq:lambda} is satisfied. Now since $tS$ consists only of
   polynomials in $t$, equivalently in $t-a$, we have that the space 
   $\frac 1{(t-a)^n}tS$ contains $1$ and is included in $L(nP_a)$
   where $P_a$ is the place at $a$.
\end{proof}

\begin{proof}[Proof of Theorem~\ref{thm:g=0}]
  Applying Proposition~\ref{prop:goodP}, we may suppose $1\in S\subset
  L(nP)$ and, without loss of generality, that $P$ is the place at
  infinity over $K(t)$: in other words, $S$ consists of a space of
  polynomials, of degree at most $n$, and including constants.
  The space $S$ must contain a polynomial of degree $n$, otherwise,
  because $\dim S=n$, $S$ would be equal to the space $L((n-1)P)$ and
  we would have $\dim S^2=2\dim S-1$, contradicting $\gamma=1$.
  Since $S$ contains constants we have that the set of degrees $d(S)$
  of the elements of $S$
  is included in the arithmetic progression $\{0,1,\ldots ,n\}$, and
  since we may find at most $2n$ different degrees in $S^2$,
  Proposition~\ref{prop:A+A=2A} implies that 
  \TabPositions{15mm}
  \begin{enumerate}
  \item either \tab $d(S)=\{0,1,2,\ldots n-3,n-2,n\}$,
  \item or    \tab $d(S)=\{0,2,3,\ldots n-2,n-1,n\}$.
  \end{enumerate}
  In case 1, we have that $S$ contains as a subspace the space of all
  polynomials of degree at most $n-2$, and also a polynomial of degree
  $n$. This gives the existence of the basis of type $(i)$ mentioned
  by the theorem. 
  It remains to deal with case 2 for which there exists a basis of $S$
  of the form 
  $$1,p_2,p_3,\ldots ,p_n$$
where $p_i$ is a polynomial of degree $i$ in the variable $t$.
Consider the sequence of subspaces 
$$S_1=K\subset S_2\subset\cdots
\subset S_n=S$$
where $S_i= S_{i-1}+Kp_i$ for $i\geq 2$.
We shall prove by induction on $k$ that for $k=3,4,\ldots ,n$, the space
$S_k$ has, possibly after changing the variable $t$, a basis of the
form $1,(t+\alpha)t,\ldots, (t+\alpha)t^{k-1}$, yielding the desired
basis of $S$ for $k=n$.
Write the Euclidean division of $p_3$ by $p_2$, 
$$p_3=(t+a)p_2 + bt+c$$
where we have set the leading coefficients of $p_2$ and $p_3$ equal to
$1$. By replacing if needed be $p_2$ by $p_2+b$ and $p_3$ by $p_3+ab-c$
we see that we may suppose that $p_2$ divides $p_3$. Without loss of
generality (change the variable $t$ to $t-\beta$, $\beta\in K$), we
may suppose that one of the roots of $p_2$ is $0$, so that
$p_2=(t+\alpha)t$ for some constant $\alpha$, and we have that $S_3$
has a basis of the required form, possibly after adding to $p_3$ a
scalar multiple of $p_2$.

Suppose now that $S_k$ has a basis of the required form, $3\leq k\leq n-1$, and
consider $S_{k+1}=S_k+Kp_{k+1}$. Without loss of generality suppose
$p_{k+1}$ has no constant term, i.e. is divisible by $t$ (replacing
$p_{k+1}$ by $p_{k+1}+c$, $c\in K$, does not change the space $S_{k+1}$).
Let $T_k$ be the subspace
of $S^2$ consisting of all polynomials in $t$ of degree at most
$k+1$. Now the set of degrees of $S^2$ is
$0,2,3,\ldots,2n$, which implies that $T_k$ cannot be equal to the
whole space of polynomials of degree at most $k+1$ so that $\dim T_k\leq k+1$.
Notice also that $T_k$ contains
\begin{equation}
  \label{eq:Sk}
  1,(t+\alpha)t,(t+\alpha)t^2,\ldots ,(t+\alpha)t^{k-1}
\end{equation}
which are all in $S_k$ by the induction hypothesis, and $T_k$
contains also 
$$(t+\alpha)^2t^{k-1} = (t+\alpha)t \times (t+\alpha)t^{k-2}.$$
Since $\dim T_k\leq k+1$, a basis of $T_k$ is therefore given by \eqref{eq:Sk} together with $(t+\alpha)^2t^{k-1}$. Now $p_{k+1}\in T_k$, so that it decomposes over the above basis,
and since $p_{k+1}$ has no constant term, we have just proved that
it is a multiple of $(t+\alpha)t$, which shows the existence of a basis of $S_{k+1}$ of the required form.
\end{proof}

Theorem~\ref{thm:g=0} shows in particular that there always exists a
valuation $v$, for which the set of valuations $v(S)$ of a space of
genus $0$ and combinatorial genus $1$ is an arithmetic
progression with a missing element (after the first or last position).
In contrast, the set of valuations for an arbitrary $v$ will typically
be an arithmetic progression. 
We now make the remark that when $g=1$ and $\gamma=1$, 
there also always exists a valuation $v$ for which $v(S)$ is an arithmetic
progression with a missing element.

Denote by $\sim$ the linear equivalence of divisors, and recall that
$G \sim H$
means that $L(G) = fL(H)$ for some function $f$.

\begin{lemma}\label{lem:one_point}
  Let $E$ be an elliptic curve and $G$ be a divisor on $E$ of degree $d$.
  Then, there exists a point $R$ of $E$ such that $G \sim dR$.
\end{lemma}

\begin{proof}
  Let $G = r_1 P_1 + \cdots + r_s P_s$. Denote by $\oplus$ the group law
  on the elliptic curve. Let $P = r_1P_1 \oplus \cdots \oplus r_s P_s$.
  From \cite[Proposition III.3.4]{Silverman}, we get
  $
  G - dO \sim P-O
  $
  where $O$ is the zero element of the group of points of $E$.
  
  Over an algebraically closed field the group of points of $E$ is divisible
  (see \cite[Theorem 4.10(a)]{Silverman}),
  hence there exists $R \in E$ such that $P = dR$. Therefore:
  $$
  G-dO \sim d(R-O)\quad \Longrightarrow \quad G \sim dR.
  $$
\end{proof}

\begin{rema}
  According to the proof, the point $R$ may not be unique since it can
  be replaced by the point $R \oplus T$ where $T$ is a $d$--torsion
  point. Thus, if $d$ is prime to the characteristic of $K$ there are
  $d^2$ possibilities for $R$.
\end{rema}

Consequently, $S$ is of the form $f L(nR)$
for some place $R$ and some nonzero function $f$. It is well--known
that the sequence of valuations of a space $L(nR)$ is $\{0, -2, -3, \ldots, -n\}$,
which can be easily derived from the Riemann-Roch
Theorem. Multiplication by $f$ only translates the sequence of valuations.

\section{{Function fields over non-algebraically closed
  fields}}\label{sec:perfect}

{
In this section we generalise Theorems~\ref{thm:freiman_field},~\ref{thm:genus0} and~\ref{thm:main} to
non-algebraically closed, perfect fields $K$.
Recall that a field $K$ is called {\em perfect} if all
  algebraic extensions of $K$ are separable.}

{
\begin{thm}\label{thm:freiman_perfect}
    Let $K$ be a perfect
  field and let $F \supseteq K$ be
  an extension field of $K$ such that $K$
  is algebraically closed in $F$. Let $S$ be a $K$-vector
  subspace of $F$ of finite dimension and of transcendence degree $d$.
  Then
  $$\dim S^2\geq (d+1)\dim S -d(d+1)/2.$$
\end{thm}
}

{
\begin{proof}
  Let $K'$ be an algebraic closure of $K$ and $F'=K'(S)$ be defined
  inside the algebraic closure of $F$.  It holds that any $K$-linearly
  independent elements of $F$ are also $K'$-linearly independent in
  $F'$.  (\cite[Proposition III.6.1]{Stichtenoth}\footnote{The context
    of the proposition is that of functions fields of one variable,
    but its proof applies verbatim to arbitrary field extensions.}).
  Therefore $\dim_KS=\dim_{K'}K'S$ and $\dim_KS^2=\dim_{K'}K'S^2$, and
  the theorem is proved by
  arguing that the transcendence degree of $S$ is the same over $K'$
  as over $K$, and that therefore
$$\dim_{K'}K'S^2\geq (d+1)\dim_{K'}K'S -d(d+1)/2$$
by applying Theorem~\ref{thm:freiman_field}.
\end{proof}
}

{
We remark that the above proof has only used that
any finite extension of $K$ is generated by a single element,
so that Theorem~\ref{thm:freiman_perfect} actually holds 
in this somewhat more general case.
}

{
\begin{thm}\label{thm:perfect}
  Let $K$ be a perfect field, algebraically closed in an extension
  field $F$.
  Let $S \subseteq F$, $1\in S$, be a space of finite dimension $n$
  and combinatorial genus $\gamma$. 
  \begin{enumerate}
  \item If $n\geq 3$ and $\gamma=0$, then $S$ has genus $0$ and
    $S=L(D)$ for $D$ a divisor of degree $n-1$.
  \item If $n\geq 4$ and $\gamma=1$ 
   then $S$ has genus $0$ or $1$. Moreover, 
  \begin{itemize}
    \item[(i)]  if $S$ has genus $1$  then $S=L(D)$ for $D$ a divisor of degree $n$,
    \item[(ii)] if $S$ has genus $0$, then $S$ is a subspace of codimension
      $1$ inside a space $L(D)$ for $D$ a divisor of degree $n$.
  \end{itemize}
  \end{enumerate}
 \end{thm}
}

{
Before proving Theorem~\ref{thm:perfect}, let us remind the reader of
some basic facts concerning algebraic extensions of function fields
that we need to call upon. We refer the reader to
\cite[Chapter III]{Stichtenoth} for more background.
Let $F'/K'$ be an algebraic
extension of $F/K$, meaning that $F'\supset F$ is an algebraic
extension and that $K'\supset K$.
Recall that if $P$ is a place of $F$ and $P'$ a place of $F'$ such
   that $P=F\cap P'$, $P$ is said to {\em lie under} $P'$ and $P'$ to {\em lie
   over} $P$. One writes $P'|P$ to mean that $P'$ lies over $P$.
   For any place $P$ of $F$,  there always exists at least a place
   $P'$ over $P$ (\cite[Proposition III.1.7]{Stichtenoth}) ,
   and for any such $P$ and $P'$ there exists 
   (\cite[Proposition III.1.4]{Stichtenoth}) an
   integer $e=e(P'|P)\geq 1$ such that $v_{P'}(x)=e\cdot v_P(x)$ for any
   $x\in F$.  The positive integer $e(P'|P)$ is called the
   {\em ramification index} of $P'$ over $P$. The {\em conorm}, with respect
   to $F'/F$, of a place
   $P$ of $F$ is defined as the divisor:
   $$Con_{F'/F}(P) = \sum_{P'|P}e(P'|P)P'.$$
   The conorm extends to divisors $D=\sum_P\alpha_PP$  of $F$ through
   the formula
   $$Con_{F'/F}(D) = \sum_P\alpha_P Con_{F'/F}(P).$$
}
{
  \begin{proof}[Proof of Theorem~\ref{thm:perfect}]
  Without loss of generality assume $F=K(S)$.
  Let $K'$ be the algebraic closure of $K$ and let $F'=K'(S)$ be
  defined inside the algebraic closure of $F$. 
  We note that such an extension $F'/F$ is unramified (\cite[Theorem
  III.6.3(a)]{Stichtenoth}), meaning that $e(P'|P)=1$ for any place
  $P$ in $F$ and any $P'$ above it.\\
  \indent
  As remarked at the end
  of the proof of Theorem~\ref{thm:freiman_perfect}, $K$-linearly
  independent elements of $F$ are also $K'$-linearly independent in
  $F'$,
  therefore $\dim_KS=\dim_{K'}K'S$ and $\dim_KS^2=\dim_{K'}(K'S)^2$,
  so that $K'S$ has combinatorial genus $\gamma$. Next,
  Theorems~\ref{thm:genus0} (case $\gamma=0$) and~\ref{thm:main} (case
  $\gamma=1$) apply to $K'S$ in the extension $F'/K'$.
  \newline
  \indent Recall (Definition~\ref{def:DU}) that 
   $$D_S =\sum_{P,\, \text{place of}\, F}-\min v_P(S)P$$
   is such that $L(D_S)$ is the Riemann-Roch space in $F$ of smallest
   dimension that contains~$S$. By \cite[Theorem
   III.6.3(b)]{Stichtenoth}, $F/K$ and $F'/K'$ have the same genus
   $g$. It remains therefore only to prove that
   \begin{equation}
     \label{eq:dimL(D_S)}
     \dim_KL(D_S) = \dim_{K'}L(D_{K'S}).
   \end{equation}
   Let $P$ be any place in the support of $D_S$.
   For any
   place $P'$
   above $P$ we have $$v_{P'}(s) = e(P'|P)v_P(s)=v_P(s)$$ (since $F'/F$
   is unramified), therefore
   $P'$ appears in the support of $D_{K'S}$. Furthermore,
   $v_{P'}(s)=v_P(s)$ for every $s\in S$, so that
   $$\min_{s\in S}v_P(s) = \min_{s\in S}v_{P'}(s) = \min_{x\in
     K'S}v_{P'}(x)$$
   since any $K'$-linear combination of elements of $S$ has a
   $P'$-valuation at least equal to $\min_{s\in S}v_{P'}(s)$.
   Therefore, the coefficient in $D_{K'S}$ of every place $P'$ above
   $P$ equals exactly the coefficient of $P$ in $D_S$.\\
   \indent
   Since any
   place $P'$ of $F'$ has a unique place $P$ lying under it in $F$, we
   deduce that we have
   $$Con_{F'/F}(D_S) = D_{K'S}$$
   from which \eqref{eq:dimL(D_S)} follows by
   \cite[Theorem III.6.3(d)]{Stichtenoth}.
   \end{proof}
}

{
We conclude by remarking that when $K$ is not algebraically closed, statement 1 of
Theorem~\ref{thm:perfect} is the correct generalisation of
Theorem~\ref{thm:genus0}. Indeed, there exist spaces $S$ of combinatorial
genus $0$ in extensions $F/K$, where $K$ is algebraically closed in
$F$, and such that $S$ does not have a basis in geometric
progression. One such example, given in \cite{bszVosper}, is obtained
by considering the field $F=\Q(x,y)$, where $\Q$ denotes the rational
field, and $y$ is algebraic over
$\Q(x)$ such that $y^2+x^2+1=0$. We have that $\Q$ is
algebraically closed in $F$, and in the extension $F/\Q$, the space
$S$ generated by $1,x,y$ has combinatorial genus $0$ but can be seen
not to have a basis in geometric progression. The space $S$ is however
equal to a Riemann-Roch space $L(P)$, where $P$ is a place of degree
$2$. When one extends the base field $\Q$ to the complex field $\C$, 
we have that $\C S$ has the basis $t^{-1},1,t$, where
$t=x+iy$,$t^{-1}=-x+iy$. Hence $\C S= L(P_0+P_{\infty})$, where $P_0$
and $P_\infty$ are the places at $0$ and at $\infty$ in $\C F=\C(t)$,
and are the two places that lie above $P$ in $\C F$.
}

{Finally, we remark that the argument spelt out in the
proof of Theorem~\ref{thm:perfect} shows that if
Conjecture~\ref{conjecture} holds, then it also holds for perfect
base fields.}

\section*{Acknowledgements}
The authors are supported by French {\em Agence nationale de la recherche} grant ANR-15-CE39-0013-01 {\em Manta}.

\bibliographystyle{abbrv}
\bibliography{freiman}

\end{document}